\documentclass[a4paper, 12pt]{amsart}
\usepackage{amsmath,amssymb,amsthm}
\usepackage{amscd}
\usepackage{array,enumerate}

\newtheorem{Thm}{Theorem}[section]

\newtheorem{Thmint}{Theorem}[section]

\theoremstyle{definition}
\newtheorem{Rem}[Thm]{Remark}
\newtheorem*{Remint}{Remark}

\newcommand{\Cs}{C$^\ast$}

\newcommand{\sd}{^{\ast\ast}}

\newcommand{\id}{\mbox{\rm id}}

\newcommand{\rc}{\mathop{\rtimes _{\mathrm r}}}
\newcommand{\votimes}{\mathop{\bar{\otimes}}}

\newcommand{\rca}[1]{\mathop{\rtimes _{{\mathrm r}, #1}}}

\newcommand{\IC}{\mathbb C}
\newcommand{\IF}{\mathbb F}
\newcommand{\IK}{\mathbb K}

\newcommand{\IN}{\mathbb N}
\newcommand{\IR}{\mathbb R}
\newcommand{\IT}{\mathbb T}
\newcommand{\IZ}{\mathbb Z}

\newcommand{\cM}{\mathcal M}

\newcommand{\acts}{\curvearrowright}

\newcommand{\ad}{\mathrm{Ad}}

\DeclareMathOperator{\supp}{supp}

\DeclareMathOperator{\Ped}{Ped}

\DeclareMathOperator{\bigfp}{\lower0.25ex\hbox{\LARGE $\ast$}}

\title[Amenable actions on stably finite simple \Cs-algebras]
{Every countable group admits amenable actions on stably finite simple \Cs-algebras}
\author{Yuhei Suzuki}
\address{Department of Mathematics, Faculty of Science, Hokkaido University,
Kita 10, Nishi 8, Kita-Ku, Sapporo, Hokkaido, 060-0810, Japan}

\dedicatory{Dedicated to Professor Yasuyuki Kawahigashi
on the occasion of his 60th birthday.}
\email{yuhei@math.sci.hokudai.ac.jp}
\begin{document}
\maketitle
\begin{abstract}
We give the first examples of (non-amenable group) amenable actions on
stably finite simple \Cs-algebras.
More precisely, we give such actions for any countable group in an explicit way.
The main ingredients of our construction
are the full Fock space of the regular representation and a trace-scaling automorphism.
\end{abstract}
\section{Introduction}
Recently amenability of \Cs-dynamical systems
attracted much attention.
This is a consequence of the discovery of
amenable \Cs-dynamical systems on \emph{simple} \Cs-algebras \cite{Suzeq}.
The formulation and characterizations of amenability of (non-commutative) \Cs-dynamical systems
had been unclear for a long time,
but this problem was recently settled in \cite{OS} (see also \cite{BC}, \cite{BEW2}, \cite{Suz21}).
Here we only recall the definition in the discrete group case.
An action $\alpha \colon \Gamma\curvearrowright A$
of a discrete group $\Gamma$
on a \Cs-algebra $A$ is \emph{amenable} \cite{AD}
if and only if there exists a $\Gamma$-equivariant conditional expectation
$\ell^\infty(\Gamma) \votimes A\sd \rightarrow 1 \otimes A\sd$
when $\ell^\infty(\Gamma) \votimes A\sd$ is equipped with
the diagonal $\Gamma$-action of the left translation action and $\alpha\sd$.
For useful characterizations and basic properties of
amenability of \Cs-dynamical systems,
we refer the reader to \cite{AD}, \cite{BEW2}, \cite{OS}.
A recent highlight around this subject is
the successful classification results \cite{Suz21}, \cite{GS1}, \cite{GS2}
of certain amenable actions on \emph{purely infinite simple} \Cs-algebras.

As an application of new characterizations of amenability,
Ozawa and the author have established
a powerful method to produce amenable actions
on simple \Cs-algebras (see \cite{OS}, Theorem 6.1).
However, in all previously known constructions \cite{Suzeq}, \cite{OS} of amenable actions (of non-amenable groups)
on simple \Cs-algebras,
the resulting underlying \Cs-algebras end up being \emph{purely infinite}.
Here we recall that the class of \emph{classifiable} simple \Cs-algebras (see e.g., the recent survey \cite{Win})
splits into two classes: \emph{stably finite} \Cs-algebras and purely infinite \Cs-algebras.
As K-theoretic invariants of a \Cs-algebra have a much richer structure in the former class,
it is a natural and important question to ask
if there also exists an amenable action on a stably finite simple \Cs-algebra.
In this article we settle this question in the non-unital case.
More precisely, we show the following theorem.
\begin{Thmint}\label{Thmint:Main}
Every countable group $\Gamma$
admits an amenable action
on a stably finite simple separable nuclear \Cs-algebra.
\end{Thmint}
To show the stable finiteness of the resulting (simple, non-unital) \Cs-algebra,
we construct a \emph{proper tracial weight} on it.
Here we recall that a tracial weight on a \Cs-algebra
is said to be \emph{proper}
if it is densely defined, lower semi-continuous, and nonzero.
We refer the reader to Chapter 5 of \cite{Pedbook}
for basic facts on tracial weights on \Cs-algebras.

On this opportunity, we also record the following generalization
of Corollary 6.4 in \cite{OS}, originally shown for the free groups,
to exact locally compact groups satisfying the strong Baum--Connes conjecture (\cite{MN1}, page 304). 
We recall that
groups with the Haagerup property satisfy the
strong Baum--Connes conjecture \cite{HK}.
\begin{Thmint}\label{Thmint:Oinfty}
Let $G$ be an exact locally compact second countable group.
Then $G$ admits an amenable action $\alpha$ on the Cuntz algebra $\mathcal{O}_\infty$.
Moreover, when $G$ satisfies the strong Baum--Connes conjecture,
one can arrange $\alpha$ to be KK$^G$-equivalent to $\mathbb{C}$.
\end{Thmint}

\begin{Remint}
(1):  The author is informed from G\'abor Szab\'o
that he, together with James Gabe, has independently proved the statement for
exact groups with the Haagerup property by a different method:
Their proof is based on their classification result (the existence theorem) \cite{GS2} and Corollary 6.3 of \cite{OS}.

\vspace{5pt}
\noindent
(2): The exactness assumption in Theorem \ref{Thmint:Oinfty}
is necessary by Corollary 3.6 of \cite{OS}.
\end{Remint}

Throughout the article, all \Cs-algebras are assumed to be nonzero.

\vspace{10pt}
For basic facts on \Cs-algebras and discrete groups,
we refer the reader to Chapters 2 to 4 of the book \cite{BO}.
For basic facts on Pedersen ideals, we refer the reader to
Chapter 5.6 of the book \cite{Pedbook}.

\subsection*{Notations}
For a Hilbert space $\frak{H}$,
denote by $\IK(\frak{H})$ the \Cs-algebra of compact operators on $\frak{H}$.
When $\frak{H}=\ell^2(\mathbb{N})$, we simply denote by $\IK$.

For a \Cs-algebra $A$, denote by $\cM(A)$ its multiplier algebra.
For an isometric element $v$ in $\cM(A)$,
denote by $\ad_v\colon A \rightarrow A$
the endomorphism given by
\[\ad_v(a):=vav^\ast,\quad a\in A.\]

For a \Cs-algebra $A$, denote by ${\rm Ped}(A)$ the Pedersen ideal of $A$ (see \cite{Pedbook}, Theorem 5.6.1).

For \Cs-algebras $A_i$ and
a proper tracial weight $\tau_i$ on $A_i$, $i=1, 2$,
denote by $\tau_1 \otimes \tau_2$
the unique proper tracial weight on the minimal tensor product $A_1 \otimes A_2$
which coincides with $\tau_1|_{\Ped(A_1)} \odot \tau_2|_{\Ped(A_2)}$ on the algebraic tensor product $\Ped(A_1) \odot \Ped(A_2) \subset \Ped(A_1 \otimes A_2)$.
(Cf.~\cite{Pedbook}, Proposition 5.6.7.)

\section{Proofs}
\begin{proof}[Proof of Theorem \ref{Thmint:Main}] 
The Higman--Neumann--Neumann embedding theorem \cite{HNN} states 
that any countable group embeds into a finitely generated group.
It is clear that the restriction of a discrete group amenable action to a subgroup
is again amenable.
Therefore it suffices to show the statement
for finitely generated groups.

Let $\Gamma$ be a non-trivial finitely generated group.
Take a finite subset $S\subset \Gamma$
with $\bigcup_{n=1}^\infty S^n =\Gamma$, $|S|>1$.
Take a simple separable nuclear \Cs-algebra $A$,
an automorphism $\rho \colon A \rightarrow A$,
and a proper tracial weight $\tau_A$ on $A$
satisfying
\[\tau_A \circ \rho = \frac{1}{|S|} \tau_A.\]
(For instance,
define $A:= \IK \otimes (\bigotimes_{\mathbb{N}} \mathbb{M}_{|S|}(\mathbb{C}))$.
Take any isomorphisms
\[\rho_1 \colon \IK \rightarrow \IK \otimes \mathbb{M}_{|S|}(\mathbb{C}), \quad
\rho_2\colon \bigotimes_{\mathbb{N}} \mathbb{M}_{|S|}(\mathbb{C}) \rightarrow
\bigotimes_{n=2}^\infty \mathbb{M}_{|S|}(\mathbb{C}),\]
and
define an automorphism
$\rho \colon A \rightarrow A$
to be $\rho:= \rho_1 \otimes \rho_2$.
Then $\rho$ together with any proper tracial weight $\tau_A$ on $A$ satisfies the foregoing equality.)

Consider the Hilbert space $\frak{H}:= \bigoplus_{n \in \IN} \ell^2(\Gamma)^{\otimes n}$.
We define a unitary representation
$u \colon \Gamma \curvearrowright \frak{H}$
to be
\[u_g := \bigoplus_{n\in \IN} \lambda_g^{\otimes n}, \quad g\in \Gamma.\]
Here $\lambda \colon \Gamma \curvearrowright \ell^2(\Gamma)$
denotes the left regular representation of $\Gamma$.
Next, for each $g\in \Gamma$,
define a map $v_g \colon \frak{H} \rightarrow \frak{H}$
to be
\[v_g(\xi):= \delta_g \otimes \xi, \quad \xi \in \frak{H}.\]
This gives a family $(v_g)_{g\in \Gamma}$ of isometric operators with pairwise orthogonal ranges.
(Cf.~the full Fock space and the left creation operators.)

Put $B:= c_0(\Gamma) \otimes \IK(\frak{H}) \otimes A$.
Let $\beta \colon \Gamma \curvearrowright B$
be the action given by
\[\beta_g := {\rm L}_g \otimes \ad_{u_g \otimes 1_A}, \quad g\in \Gamma.\]
Here ${\rm L} \colon \Gamma \curvearrowright c_0(\Gamma)$
denotes the left translation action
and $1_A$ denotes the unit of $\cM(A)$.
For each $g\in \Gamma$,
define an endomorphism $\nu_g$ on $\IK(\frak{H}) \otimes A$
to be $\nu_g := \ad_{v_g} \otimes \rho$.
Since the isometric elements $(v_g)_{g\in \Gamma}$ have
pairwise orthogonal ranges, one can define an injective endomorphism $\sigma \colon B \rightarrow B$
by the formula
\[\sigma(b):= \sum_{g\in \Gamma} \chi_{g S} \otimes \nu_g(b_g), \quad b\in B.\]
Here we identify $B$ with
$c_0(\Gamma, \IK(\frak{H}) \otimes A)$ in the obvious way.

Let $\tau_\Gamma$ denote the proper tracial weight on $c_0(\Gamma)$
defined by the counting measure on $\Gamma$.
Fix a proper tracial weight $\tau_\mathfrak{H}$ on $\IK(\mathfrak{H})$.
On the \Cs-algebra $B$, consider the proper tracial weight
\[\tau_B := \tau_\Gamma \otimes \tau_\mathfrak{H} \otimes \tau_A.\]
Since $\IK(\mathfrak{H})\otimes A$ is simple, $\tau_{\mathfrak{H}} \otimes \tau_A$, hence $\tau_B$, is faithful.

Since 
\[(\tau_\mathfrak{H} \otimes \tau_A) \circ \nu_g = \frac{1}{|S|}\tau_\mathfrak{H} \otimes \tau_A\]
for all $g\in\Gamma$, one has
$\tau_B \circ \sigma = \tau_B$.

Next we define a \Cs-algebra $\mathfrak{B}$ to be the inductive limit
of the inductive system
\[
 \begin{CD}
 B @>{\sigma}>> B @>{\sigma}>> B @>{\sigma}>> \cdots 
 \end{CD}
\]
of \Cs-algebras.
For each $n\in \IN$, let \[\theta_n \colon B \rightarrow \mathfrak{B}\]
denote the canonical embedding from the $n$-th $B$ into $\mathfrak{B}$.
Since $\sigma$ preserves $\tau_B$,
one has a proper tracial weight $\tau_{\mathfrak{B}}$ on $\mathfrak{B}$
satisfying
\[\tau_{\mathfrak{B}} \circ \theta_n = \tau_B\]
for all $n\in \IN$.
Since $\tau_B$ is faithful, so is $\tau_{\mathfrak{B}}$.

We next show that $\sigma \colon B \rightarrow B$ is $\Gamma$-equivariant.
To see this, note that
\[u_g v_h = v_{gh} u_g\]
for all $g, h\in \Gamma$.
This implies
\[\ad_{u_g \otimes 1_A} \circ \nu_h= \nu_{gh} \circ \ad_{u_g\otimes 1_A}.\]
Then, for any $g\in \Gamma$ and any $b\in B$, we have
\begin{align*}
\beta_g(\sigma(b))&= \sum_{h\in \Gamma} \chi_{ghS} \otimes \ad_{u_{g} \otimes 1_A}(\nu_h(b_h))\\
&=\sum_{h\in \Gamma}\chi_{ghS} \otimes \nu_{gh}(\ad_{u_g \otimes 1_A}(b_h))\\
&=\sum_{h\in \Gamma}\chi_{hS} \otimes \nu_{h}(\ad_{u_g \otimes 1_A}(b_{g^{-1}h}))\\
&=\sigma(\beta_g(b)).
\end{align*}
Hence $\sigma$ is $\Gamma$-equivariant.
Therefore one has an action
$\gamma \colon \Gamma \curvearrowright \mathfrak{B}$
satisfying
\[\gamma_g(\theta_n(b))=\theta_n(\beta_g(b))\]
for all $g\in \Gamma$, $n\in \IN$, $b\in B$.

We next define an automorphism $\alpha \colon \mathfrak{B} \rightarrow \mathfrak{B}$
to be
\[\alpha(\theta_n(b)):= \theta_{n+1}(b), \quad n\in \IN, b\in B.\]
It is not hard to see that $\alpha$ is indeed a (well-defined) automorphism of $\mathfrak{B}$.
Clearly $\alpha$ is $\Gamma$-equivariant and preserves $\tau_{\mathfrak{B}}$.
Observe that for any $n\in \IN$, as $\theta_{n+1}\circ \sigma =\theta_n$, one has
$\alpha^{-1}\circ \theta_n = \theta_n \circ \sigma$.

Consider the reduced crossed product \Cs-algebra
\[C:=\mathfrak{B}\rca{\alpha} \mathbb{Z}.\]
Let $w\in \cM(C)$ denote the canonical implementing unitary element of $\alpha$.
Let
\[E \colon C=\mathfrak{B}\rca{\alpha} \mathbb{Z} \rightarrow \mathfrak{B}\]
denote the canonical conditional expectation.
Since $\alpha$ preserves $\tau_{\mathfrak{B}}$,
one has a faithful proper tracial weight $\tau_C:=\tau_{\mathfrak{B}} \circ E$ on $C$.
In particular $C$ is stably finite.
Since the two actions $\gamma$ and $\alpha$ commute,
one has an action
$\eta \colon \Gamma \curvearrowright C$
satisfying
\[\eta_g(x w^n)=\gamma_g(x)w^n\quad {\rm~for~} x\in \mathfrak{B}, n\in \IZ, g\in \Gamma.\]
We show that $(C, \eta)$ possesses the desired properties.

To see the amenability of $\eta$,
let us consider the gauge action $\zeta \colon \IT \curvearrowright C$.
(That is, the action given by
$\zeta_z(x w^n) := z^n x w^n$
for $z\in \IT:=\{s \in \IC: |s|=1\}$, $x\in \mathfrak{B}$, $n\in \IZ$.)
Note that $\zeta$ commutes with $\eta$
and that the fixed point algebra $C^{\IT}$
is equal to $\mathfrak{B}$.
Since $(B, \beta)$ is (strongly) amenable,
the action $\gamma$
is amenable.
Now Theorem 5.1 of \cite{OS} implies the amenability of $\eta$.
(An alternative proof: One can directly show
the nuclearity of $C \rca{\eta} \Gamma \cong(\mathfrak{B}\rca{\gamma} \Gamma)\rc \IZ$.
Then Theorem 4.5 of \cite{AD} implies
the amenability of $\eta$.)

We now prove that the \Cs-algebra
$C$ has the desired properties.
In fact all properties other than the simplicity of $C$ are clear from the construction.
To show the simplicity of $C$,
we first show
that $\mathfrak{B}$ has
no proper (two-sided, closed) ideal $I$
which is $\alpha$-invariant, that is, $\alpha(I)=I$.
Let $I \subset \mathfrak{B}$ be a nonzero $\alpha$-invariant ideal.
Choose $n\in \IN$ satisfying
$I \cap \theta_n(B)\neq \{0\}$.
Observe that any ideal $J$ of $B$ is of the form
$c_0(X) \otimes \IK(\mathfrak{H}) \otimes A$
for some subset $X \subset \Gamma$.
Here we identify $c_0(X)$ with an ideal of $c_0(\Gamma)$ in the obvious way.
Choose a (non-empty) subset $X\subset \Gamma$ satisfying
\[I \cap \theta_n(B) = \theta_n(c_0(X)\otimes \IK(\mathfrak{H}) \otimes A).\]
We claim that $X=\Gamma$.
Observe that for any $b\in B$,
one has
\[\supp(\sigma(b))= \supp(b)S\]
by the definition of $\sigma$.
Since $\alpha(I)=I$,
we obtain
\[\theta_n(\sigma(c_0(X)\otimes \IK(\mathfrak{H}) \otimes A))
=\alpha^{-1}(I \cap \theta_n(B))
= \alpha^{-1}(I)\cap \theta_n(\sigma(B)) \subset I \cap \theta_n(B).\]
This proves $X S \subset X$.
By the choice of $S$,
we have $X=\Gamma$.
This yields $I=\mathfrak{B}$.

Now let $I \subset C$ be a nonzero ideal.
As $\mathfrak{B}$ has no proper $\alpha$-invariant ideal,
it suffices to show that $I \cap \mathfrak{B} \neq \{0\}$.
Since $E$ is a faithful conditional expectation, the image $J:=E(I)$ is
a nonzero algebraic ideal of $\mathfrak{B}$ with $\alpha(J)=E(w I w^\ast) =J$.
Hence $J$ is norm dense in $\mathfrak{B}$.
By Theorem 5.6.1 in \cite{Pedbook}, we have $\Ped(\mathfrak{B}) \subset J$. 
Since $\theta_1(\sigma({\rm Ped}(B)))=\Ped(\theta_1(\sigma(B))) \subset \Ped(\mathfrak{B})$,
we conclude $\theta_1(\sigma(\Ped(B))) \subset J$. 
Then, notice that the algebraic tensor product $c_c(\Gamma)\odot \IF(\mathfrak{H}) \odot \Ped(A)$ is contained in $\Ped(B)$, where $\IF(\mathfrak{H})=\Ped(\IK(\mathfrak{H}))$ denotes the $\ast$-algebra
consisting of all bounded operators of finite rank on $\mathfrak{H}$.
Hence one can find a self-adjoint element $c\in I$
with
\[E(c)=\theta_1(\sigma(r \otimes a))\]
for some positive elements $r\in c_0(\Gamma)\otimes \IK(\mathfrak{H})$
and $a\in A$ with $\|r \|=\|a\|=1$.
Note that for any $n\in \IN$, by the definition of $\sigma$,
there is a positive element $r_n \in c_0(\Gamma)\otimes \IK(\mathfrak{H})$ of norm one
satisfying 
$\sigma^n(r\otimes d) = r_n \otimes \rho^n(d)$ for all $d\in A$.
We next choose a self-adjoint element $c_0\in C$
of the form
\[c_0=\sum_{k=-N}^N \theta_{M}(b_k) w^k \quad {\rm~where~} N, M \in \IN,\quad b_k\in B, k=-N, -N+1, \ldots, N\]
satisfying
\[\theta_{M}(b_0)= \theta_1(\sigma(r \otimes a))~ ({\rm that~is}, E(c)=E(c_0)), \quad \|c-c_0\|< \frac{1}{2}.\]
Observe that 
\[\theta_1(\sigma(r \otimes a))
=\theta_{M}(\sigma^{M}(r\otimes a))= \theta_{M}(r_{M} \otimes \rho^{M}(a))\]
hence
$b_0= r_{M} \otimes \rho^{M}(a)$.

Note that the nonzero powers of $\rho$ are outer,
because they do not preserve the proper tracial weight $\tau_A$.
Therefore, by Kishimoto's theorem \cite{Kis} (see Lemma 3.2 therein),
one can find a sequence $(x_n)_{n=1}^\infty$ in $A$
satisfying
\begin{enumerate}
\item $\|x_n\|=1$ for all $n\in \IN$,
\item $\lim_{n\rightarrow \infty} \|x_n d \rho^k(x_n^\ast)\| = 0$
for all $d \in A$ and $k=1, \ldots, N$,
\item $\lim_{n\rightarrow \infty} \|x_n \rho^{M}(a) x_n^\ast\|=1$.
\end{enumerate}
Put $y_n:=\theta_{M}(r_{M} \otimes x_n)$ for $n\in \IN$.
Then for any $k=1, 2, \ldots, N$, by conditions (1) and (2), one has
\begin{align*}\lim_{n\rightarrow \infty}\|y_n \theta_{M}(b_{-k}) w^{-k} y_n^\ast\|
&= \lim_{n\rightarrow \infty}\|(r_{M} \otimes x_n) b_{-k} (r_{M+k} \otimes \rho^k(x_n^\ast))\|\\
&=0.\end{align*}
Since $c_0$ is self-adjoint,
this proves
\[\lim_{n\rightarrow \infty} y_n(c_0 - E(c_0)) y_n^\ast =0.\]
Condition (3) yields
\[\lim_{n\rightarrow \infty}\|y_n \theta_{M}(b_0) y_n^\ast\|
= \lim_{n\rightarrow \infty}\|r^3_M \otimes x_n\rho^M(a) x_n^\ast \|
=1.\]
Therefore for a sufficiently large $n\in \IN$,
we have
\[{\rm dist}(y_n E(c)y_n^\ast, I) \leq \| y_n(c-E(c))y_n^\ast\| < \frac{1}{2}<\|y_n E(c) y_n ^\ast\|.\]
This shows that the canonical map
$\mathfrak{B} \rightarrow C/I$
is not isometric.
Thus $\mathfrak{B} \cap I \neq \{0\}$.
\end{proof}
\begin{Rem}
Here we record a few observations on the foregoing construction.
Keep the same notations as in the proof of Theorem \ref{Thmint:Main}.

\vspace{5pt}
\noindent
(1): When $A$ satisfies the universal coefficient theorem of Rosenberg--Schochet \cite{RS},
so does the resulting \Cs-algebra $C$.

\vspace{5pt}
\noindent
(2): The tracial weight $\tau_C$ on $C$ is $\Gamma$-invariant.
Hence the reduced crossed product \Cs-algebra $C \rca{\eta} \Gamma$ has a faithful tracial weight.
In particular $C \rca{\eta} \Gamma$ is stably finite.
We also note that $C \rca{\eta} \Gamma$ is simple.
To see this, observe that $B \rca{\beta} \Gamma$ is isomorphic to $\IK(\ell^2(\Gamma)) \otimes  \IK(\mathfrak{H}) \otimes A$.
Since $\mathfrak{B} \rca{\gamma} \Gamma$ is isomorphism to an inductive limit of copies of $B \rca{\beta} \Gamma$,
we conclude the simplicity of $\mathfrak{B} \rca{\gamma} \Gamma$.
Now the same proof as the simplicity of $C$ proves
the simplicity of $C  \rca{\eta} \Gamma \cong (\mathfrak{B} \rca{\gamma} \Gamma) \rc \IZ$.

\vspace{5pt}
\noindent
(3): When $\Gamma$ is non-amenable, the \Cs-algebra $C$ has proper tracial weights other than the
scalar multiples of $\tau_C$.
This follows from the following general observation.
By \cite{AD79},
non-amenable groups cannot act amenably on a factor (in the von Neumann algebra sense).
Thus, when a non-amenable group has an amenable action on a \Cs-algebra,
its space of all proper tracial weights cannot be a ray.

Alternatively, one can give different proper tracial weights on $C$ as follows.
Let $\nu$ denote the averaging measure on $S^{-1} \subset \Gamma$.
Then we replace the counting measure in the definition of $\tau_B$
 by another (nonzero, $\sigma$-finite) right $\nu$-harmonic measure $\mu$ on 
$\Gamma$ (that is, $\mu \ast \nu= \mu$).
This gives a $\sigma$-invariant proper tracial weight $\tau_{B, \mu}$ on $B$.
Then, by the same way as in the original case,
$\tau_{B, \mu}$ gives an $\alpha$-invariant proper tracial weight on $\mathfrak{B}$.
As a result we obtain a proper tracial weight on $C$.
(The new proper tracial weight is no longer $\Gamma$-invariant,
unlike the original $\tau_C$.)
Note that $\Gamma$ has many right $\nu$-harmonic measures.
In fact, the space ${\rm H}^{\infty}_{\mathrm r}(\Gamma, \nu)$ of bounded right $\nu$-harmonic functions 
is already huge.
To see this, define $P_\nu \colon \ell^\infty(\Gamma) \rightarrow \ell^\infty(\Gamma)$ to be $P_\nu(f):=f\ast \nu$, $f\in \ell^\infty(\Gamma)$.
Then any cluster point of the sequence
\[\Phi_n:= \frac{1}{n}\sum_{k=1}^n P_\nu ^k \colon \ell^\infty(\Gamma) \rightarrow \ell^\infty(\Gamma),\quad n\in \IN,\]
in the point-ultraweak topology
gives a positive projection $\Phi \colon \ell^\infty(\Gamma) \rightarrow {\rm H}^\infty_{\mathrm r}(\Gamma, \nu)$.
Moreover $\Phi$ commutes with the left translation $\Gamma$-action.
Since $\Gamma$ is non-amenable,
the Choi--Effros product makes ${\rm H}^\infty_{\mathrm r}(\Gamma, \nu)$
an infinite-dimensional injective \Cs-algebra.

\vspace{5pt}
\noindent
(4): Consider the special case that one has a (continuous) flow $\Theta \colon \IR \acts A$ 
which commutes with $\rho$ and
satisfies $\tau_A \circ \Theta_t  =\kappa^t \tau_A$, $t\in \IR$, for some constant
$\kappa>1$.
(See e.g., \cite{KK}.)
Then clearly
$\Xi_t:=\id_{c_0(\Gamma)} \otimes \id_{\IK(\mathfrak{H})} \otimes \Theta_t$ and $\sigma$ commute for all $t\in \IR$.
Hence one has a flow $\Psi \colon \IR \acts \mathfrak{B}$
satisfying $\Psi_t \circ \theta_n = \theta_n \circ \Xi_t$ for all $n\in \IN$ and $t\in \IR$.
The flow $\Psi$ and the automorphism $\alpha$ commute
hence $\Psi$ induces a flow $\Omega$ on $C$.
By the definitions of $\Psi$ and $\tau_C$, one has
$\tau_C \circ \Omega_t = \kappa^t \tau_C$
for all $t\in \IR$.
In particular, the resulting \Cs-algebra $C$ is stably projectionless.
\end{Rem}

\begin{proof}[Proof of Theorem \ref{Thmint:Oinfty}]
Take an amenable action
$\beta\colon G \curvearrowright X$ on a compact metrizable space \cite{Oz}, \cite{BCL}.
Let $\gamma \colon G \curvearrowright {\rm Prob}(X)$
denote the action induced from $\beta$.
Here we equip ${\rm Prob}(X) \subset C(X)^\ast$ with the weak-$\ast$ topology.
Then it is known that $\gamma$ is again amenable (see e.g., Exercise 15.2.1 of \cite{BO} or Proposition 3.5 of \cite{OS}).
Note that ${\rm Prob}(X)$ is a metrizable compact convex subset of $C(X)^\ast$
in the weak-$\ast$ topology.
Therefore any closed amenable subgroup $L$ of $G$
has a fixed point in ${\rm Prob}(X)$.
Hence ${\rm Prob}(X)$ is $L$-equivariantly contractible.
In particular the unital inclusion $\IC \subset C({\rm Prob}(X))$
is a weak equivalence in KK$^G$ in the sense of \cite{MN1}.

We now apply Theorem 6.1 in \cite{OS}
to $(C({\rm Prob}(X)), \gamma)$.
This gives an ambient $G$-\Cs-algebra
$(A, \alpha)$ of $(C({\rm Prob}(X)), \gamma)$ satisfying
the following conditions:
\begin{enumerate}
\item
$A$ is a Kirchberg algebra,
\item $(A, \alpha)$ is amenable,
\item
the inclusion
$C({\rm Prob}(X)) \subset A$ is unital and KK$^G$-equivalent.
\end{enumerate}
From conditions (1) and (3),
one has $A \cong \mathcal{O}_\infty$ by the Kirchberg--Phillips classification theorem \cite{Kir}, \cite{Phi}.

When $G$ satisfies the strong Baum-Connes conjecture (in the sense of \cite{MN2}, page 304),
the unital inclusion 
$\mathbb{C} \subset C({\rm Prob}(X))$
and hence also the unital inclusion $\IC \subset A$
is KK$^G$-equivalent (cf.~ Proposition 4.3 and Corollary 7.3 in \cite{MN1}).
\end{proof}
\subsection*{Acknowledgements}
The author is grateful to Professor Yasuyuki Kawahigashi for giving me an opportunity
to give an intensive lecture at the University of Tokyo.
This motivated me to reconsider the problem solved in this paper (Theorem \ref{Thmint:Main}).

He is grateful to Professor Masaki Izumi for a stimulating comment
on Corollary 6.4 of \cite{OS},
which leads to obtain Theorem \ref{Thmint:Oinfty}.
He is also grateful to Professor G\'abor Szab\'o
for helpful comments on a draft of this article.

Finally I would like to thank the referee for several comments and suggestions,
which improve the readability of this article.

This work was supported by JSPS KAKENHI (Grant-in-Aid for Early-Career Scientists)
Grant Numbers JP19K14550, JP22K13924.

\end{document}